\documentclass[11pt,reqno]{amsart}

\usepackage[bookmarksnumbered,colorlinks, linkcolor=blue, citecolor=red, pagebackref, bookmarks, breaklinks]{hyperref}
\usepackage{color}
\usepackage{amsmath, amssymb, latexsym, amscd, amsthm, amsfonts, amstext}
\usepackage[mathscr]{eucal}
\usepackage{xspace}

\voffset= - 0.5in
\theoremstyle{plain}
\newtheorem{theorem}{Theorem}[section]

\newtheorem{lemma}[theorem]{Lemma}
\newtheorem{remark}[theorem]{Remark}
\newtheorem{proposition}[theorem]{Proposition}

\theoremstyle{definition}
\newtheorem{definition}[theorem]{Definition}

\def\NN{{\mathbb N} }

\def\RR{{\mathbb R} }
\def\X{{\mathcal X} }
\def\to{\rightarrow}

\def\phi{{\varphi} }

\usepackage{mathtools}
\mathtoolsset{showonlyrefs} 

\begin{document}

\title
{Neumann and Robin type boundary conditions in Fractional Orlicz-Sobolev spaces}

\author{Sabri Bahrouni}
\address[S. Bahrouni]{Department of Mathematics, Faculty of Sciences,
University of Monastir, 5019 Monastir, Tunisia}
\email{sabri.bahrouni@fsm.rnu.tn}

\author{Ariel M.  Salort}
\address[A. Salort]{Departamento de Matem\'atica, FCEyN - Universidad de Buenos Aires and
\hfill\break \indent IMAS - ºICET
\hfill\break \indent Ciudad Universitaria, Pabell\'on I (1428) Av. Cantilo s/n. \hfill\break \indent Buenos Aires, Argentina.}
\email[A.M. Salort]{asalort@dm.uba.ar}
\urladdr{http://mate.dm.uba.ar/~asalort}

\keywords{Fractional Orlicz-Sobolev spaces; Neumann and Robin problem; Three solutions; Eigenvalue problems. \\
\hspace*{.3cm} {\it 2010 Mathematics Subject Classifications}:
 46E30, 35R11, 45G05}

\begin{abstract}
In the first part of this article we deal with the existence of at least three non-trivial weak solutions of a nonlocal problem with nonstandard growth involving a nonlocal Robin type boundary condition. The second part of the article is devoted to study eigenvalues and minimizers of  several nonlocal problems for the fractional $g-$Laplacian $(-\Delta_g)^s$ with different boundary conditions, namely, Dirichlet, Neumann and Robin.
\end{abstract}
\maketitle
\tableofcontents
\numberwithin{equation}{section}

\section{Introduction}
In  the   recent years has been an increasing interest in studying non-local problems with $p-$structure due to its accurate description of models  involving anomalous diffusion.  In several branches of science have been observed some   phenomena having a non-local nature, which, nonetheless, do not obey a power-like growth law. See for instance \cite{AHK,  Cianchi, BR, FBS, KRV} and references therein.

The suitable operator to describe these kind of phenomena is the fractional $g-$Laplacian introduced in \cite{FBS} and defined as
\begin{equation} \label{g.laplacian}
(-\Delta_g)^s u:= \, \text{p.v.} \int_{\mathbb{R}^n} g\left( |D_s u|\right)\frac{D_s u}{|D_s u|} \frac{dy}{|x-y|^{n+s}},
\end{equation}
and defined in the principal value sense; here $G$ is a Young function such that $g=G'$ and $s\in (0,1)$ is a fractional parameter. The quantity $D_s u:=\frac{u(x)-u(y)}{|x-y|^s}$ is the \emph{$s-$H\"older quotient}.
 
Problems involving this operator have recently attracted some attention.  We refer the readers to \cite{ Cianchi, Cianchi2, Azroul, Sabri1, sabri2, sabri3, sabri4, FBS, Salort3, DNFBS, Salort2}. Observe that when $G(t)=t^p/p$, $p>1$, \eqref{g.laplacian} becomes the well-known \emph{fractional $p-$Laplacian operator}.

Given an open bounded domain $\Omega\subset \RR^n$ with smooth boundary ($\partial\Omega\in C^{0,1}$ is enough)  the first aim  of the present article is to study existence of nontrivial solutions of the following equation involving the nonlinearities $f$ and $h$ with homogeneous Robin \emph{boundary} condition on $\RR^n\setminus \Omega$
\begin{align} \label{eq}
\begin{cases}
(-\Delta_g)^s u+g(u)\frac{u}{|u|} =\lambda f(x,u)+\mu h(x,u)&\text{ in } \Omega\\
\mathcal{N}_g u +\beta(x)g(u)\frac{u}{|u|}=0 &\text{ in } \mathbb{R}^n \setminus \Omega.
\end{cases}
\end{align}
Here, we introduce a non-local normal derivative in this settings as
\begin{equation} \label{normal}
\mathcal{N}_g u(x):= \int_{\Omega} g\left( |D_s u|\right)\frac{D_s u}{|D_s u|} \frac{dy }{|x-y|^{n+s}}, \qquad x\in \mathbb{R}^n\setminus \bar\Omega,
\end{equation}
which can be seen as the natural generalization of the non-local derivative introduced in \cite{Dipierro}.

Nonlocal equations for the fractional $p-$Laplacian with boundary conditions involving  non-local normal derivatives have   been recently developed in the literature; see for instance \cite{Nicola, DPRS, DPS, Dipierro, Winkert, Dimitri, Bruno, W}.

Regarding existence of solutions to problem \eqref{eq} in the particular case of the  fractional $p-$Laplacian, there has been  some recent develops. In \cite{Dimitri}, under suitable conditions on the nonlinearities, the authors obtain existence of at most one positive solution by following the celebrated paper of Brezis-Oswald. The authors in \cite{ML}, for the same problem but with $\beta\equiv0$, and under suitable conditions on $f$, by using variational methods obtain existence of two positive solutions. It worths to be mention that the local counterpart of \eqref{eq} for Orlicz functions in the Dirichlet case was studied in \cite{Vilasi, Kristaly, Nguyen}. For some existence results in the nonlocal Orlicz case with Dirichlet boundary conditions see \cite{Azroul}.

Our first main scope is to provide conditions on the Young function $G$, on the nonlinearities $f$ and $h$, and over $\lambda, \mu$ and $\beta$ to ensure existence of at least three nontrivial (weak) solutions of \eqref{eq}. Our arguments are based in the celebrated result \cite{Ricceri} by B. Ricceri together with an \emph{integration by parts formula} related to the operator $(-\Delta_g)^s$.

The Young function $G=\int_0^tg(t)\,dt$ is assumed to satisfy the following growing condition
\begin{equation} \label{cond.intro} \tag{$G_1$}
1<p^-\leq \frac{tg(t)}{G(t)} \leq p^+<\infty \quad \forall t>0
\end{equation}
for fixed constants $p^\pm$. Moreover,  the following structural condition is assumed
\begin{equation} \label{G2} \tag{$G_2$}
t\mapsto G(\sqrt{t}),\ t\in[0,\infty[  \text{ is convex}.
\end{equation}
To ensure compactness we restrict ourselves to the  \emph{sub-critical case} of the fractional Orlicz-Sobolev embeddings:
\begin{equation} \label{G3} \tag{$G_3$}
 \displaystyle\int_{0}^{1}\frac{G^{-1}(\tau)}{\tau^{\frac{n+s}{n}}}d\tau<\infty\quad \text{and}\quad
\displaystyle\int_{1}^{+\infty}\frac{G^{-1}(\tau)}{\tau^{\frac{n+s}{n}}}d\tau=\infty.
\end{equation}

Here, $\lambda$ and $\mu$ are two positive real parameters in a suitable range and $\beta\in L^{\infty}(\mathbb{R}^n\backslash\Omega)$ is strictly positive. The nonlinearities $f,h\colon\Omega\times\mathbb{R}\rightarrow\mathbb{R}$  will be  suitable Carath\'eodory continuous functions assumed to belong to the class $\mathcal{A}$ defined as follows: $f\in \mathcal{A}$ if it fulfills the growth condition
\begin{equation} \tag{$f_1$} \label{f1}
|f(x,t)| \leq w(x) (1+m(|t|))\quad \text{for  }a.e \, x\in\Omega \text{ and for all}\, t\in\mathbb{R},
\end{equation}
where $w$ is a positive function such that $w\in L^{\infty}(\Omega)$ and $m=M'$, being $M$ a Young function 
decreasing essentially more rapidly than  the critical Sobolev function $G_*$, i.e., $M\prec\prec G_*$, being $G_*$ the critical function in the fractional Orlicz-Sobolev embedding (see section \ref{sec.sobolev} for details). We remark  that \eqref{f1} is fulfilled, for instance, if $|f(x,t)| \leq w(x) (1+|u|)^{q-1}$ for some $q\in (1,p^-_*)$, being $p^-_*:=\frac{np^-}{n-p^-}$.

From now on, we denote
$$
F(x,t)=\int_{0}^{t}f(x,s)ds,\qquad H(x,t)=\int_{0}^{t}h(x,s)ds, \qquad \mathcal{F}(u)=\int_\Omega F(x,u)\,dx,
$$
and we anticipate that the natural space to look for (weak) solutions of \eqref{eq} is given by (see Section \ref{sec.3} for details and motivations)
$$
\X=\left\{ u \text{ measurable} \colon \iint_{\RR^{2n}\setminus (\Omega^c)^2} G(|D_s u(x,y)|)\,d\mu + \int_\Omega G(|u|)\,dx + \int_{\RR^n\setminus \Omega} \beta G(|u|)\,dx <\infty \right\},
$$
where we have denoted $d\mu:=\frac{dx\,dy}{|x-y|^n}$.

With these preliminaries, our first result reads as follows.

\begin{theorem}\label{Three solution}
Let $G$ be a Young function satisfying the structural hypotheses \eqref{cond.intro},\eqref{G2} and \eqref{G3}, let $\beta\in L^\infty(\RR^n\setminus \Omega)$ and let $f,h \in \mathcal{A}$ such that
\begin{equation}  \tag{$F_1$} \label{pf1}
\max\left\{ \limsup_{|u|\to 0}\frac{\sup_{x\in\Omega} F(x,u)}{G(u)}, \limsup_{|u|\to+\infty}\frac{\sup_{x\in\Omega} F(x,u)}{G(u)} \right\}\leq 0,
\end{equation}

\begin{equation}  \tag{$F_2$} \label{pf2}
\sup_{u\in \X}\displaystyle\int_{\Omega}F(x,u)dx>0.
\end{equation}
Then, if we set
$$\delta=\inf\left\{\frac{\mathcal{J}(u)}{\mathcal{F}(u)}\colon u\in \X,\ \mathcal{F}(u)>0\right\},$$
where
$$
\mathcal{J}(u):=\iint_{\mathbb{R}^{2n} \setminus (\Omega^c)^2} G(|D_s u|)\,d\mu+ \int_{\Omega}G(|u|)dx +\int_{\mathbb{R}^{n}\backslash\Omega}\beta G(|u|)dx,
$$
for each compact interval $[a,b]\subset (\delta,\infty)$ there exists $\nu>0$ with the following property: for every $\lambda\in[a,b]$ and $h$, there exists $\gamma>0$ such that, for each $\mu\in[0,\gamma]$,    problem \eqref{eq} has at least three weak solutions whose norms in $\X$ are less than $\nu$.
\end{theorem}

We also prove the following result characterizing the geometry involved in the class of admissible nonlinearities.

\begin{theorem} \label{teo2}
Let $G$ be a Young function satisfying \eqref{cond.intro},\eqref{G2} and \eqref{G3}, let $\beta\in L^\infty(\RR^n\setminus \Omega)$ and let $f,h\in \mathcal{A}$ such that
\begin{itemize}
\item[(i)] there exists a Young function $B(t)=\int_0^t b(\tau)\,d\tau$ such that $\frac{tb(t)}{B(t)}\leq b^+<p^-$ and  $B\prec\prec G$, and a constant $c_1>0$ for which
$$
F(x,t)\leq c_1(1+B(t)) \quad \text{for all }(x,t)\in \Omega\times\RR;
$$

\item[(ii)] there exist a constant $c_2>0$, $\tau_1>0$ and a Young function $D(t)=\int_0^t d(\tau)\,d\tau$ such that $p^+<d^-\leq \frac{td(t)}{D(t)}$ and $G\prec\prec D$ for which
$$
F(x,t)\leq c_2 D(t) \quad \text{for all }(x,t)\in \Omega\times[-\tau_1,\tau_1];
$$
\item[(iii)] there exists $\tau_2 \in \RR\setminus \{0\}$ such that
$$
F(x,\tau_2)>0 \quad \text{ and } F(x,t)\geq 0 \quad \text{for all } (x,t)\in \Omega\times [0,\tau_2].
$$
\end{itemize}
Then there exists $\delta>0$ such that for every compact interval $[a,b]\subset (\delta,\infty)$ there exists a real number $\nu$ such that, for every $\lambda\in [a,b]$ and every continuous function $h$ there exists $\gamma>0$ such that, for each $\mu \in [0,\gamma]$, then  problem \eqref{eq} has at least three weak solutions whose norms in $\X$ are less than $\nu$.
\end{theorem}

We remark that the class of admissible  nonlinearities in Theorem \ref{Three solution} includes perturbations of powers and concave-convex type combinations, among other. See Section \ref{sec.ejemplos} for further examples.

\medskip

Very close to \eqref{eq}, as a second aim, we will study eigenvalues and minimizers of several nonlocal problems with non-standard growth involving different boundary conditions. For the case of powers, that is, for fractional $p-$Laplacian type operators, the Dirichlet case was studied for instance in \cite{LL,SV}, for the Neumann case see for instance \cite{DPS, ML}, the Robin case was dealt in \cite{K}. For general Orlicz functions and Dirichlet boundary conditions we refer to \cite{Salort2}.

To be more precise, we  consider the following Dirichlet eigenvalue problem
\begin{align} \label{eq.d}
\begin{cases}
(-\Delta_g)^s u + g(|u|)\frac{u}{|u|}=\lambda g(|u|)\frac{u}{|u|} &\text{ in } \Omega\\
u=0 &\text{ in } \RR^n \setminus \Omega,
\end{cases}
\end{align}
the following Neumann problem in terms of the nonlocal normal derivative $\mathcal{N}_g$
\begin{align} \label{eq.n}
\begin{cases}
(-\Delta_g)^s u + g(|u|)\frac{u}{|u|}=\lambda g(|u|) \frac{u}{|u|} &\text{ in } \Omega\\
\mathcal{N}_g u=0 &\text{ in } \RR^n \setminus \Omega,
\end{cases}
\end{align}
the following problem, which, from a probabilistic point of view can be seen also as a Neumann eigenvalue problem (see \cite{DPS})
\begin{align} \label{eq.n2}
\begin{cases}
(-\Delta_{g})_\Omega^s u + g(|u|)\frac{u}{|u|}=\lambda g(|u|)\frac{u}{|u|} &\text{ in } \Omega\\
u\in W^{s,G}_{reg}.
\end{cases}
\end{align}
and finally, the following Robin eigenvalue problem
\begin{align} \label{eq.r}
\begin{cases}
(-\Delta_{g})^s u + g(|u|)\frac{u}{|u|}=\lambda g(|u|)\frac{u}{|u|} &\text{ in } \Omega\\
\mathcal{N}_g u + \beta g(|u|)\frac{u}{|u|}=0 &\text{ in } \RR^n \setminus \Omega.
\end{cases}
\end{align}
Here, for $0<s<1$ we have denoted the \emph{regional fractional} $g-$Laplacian   as
\begin{align*} 
(-\Delta_{g})_\Omega^s u&:=2 \,\text{p.v.} \int_{\Omega\times\Omega} g( |D_s u|) \frac{D_s u}{|D_s u|} \frac{dy}{|x-y|^{n+s}},
\end{align*}
which is naturally defined in the space
$$
W^{s,G}_{reg}(\Omega):=\left\{u\colon  \int_\Omega G(|u|)\,dx +  \iint_{\Omega\times\Omega} G\left( D_s u \right) \,d\mu<\infty\right\}.
$$

A substantial difference which contrasts with the case of powers is that, in general, eigenvalues of \eqref{eq.d},  \eqref{eq.n}, \eqref{eq.n2} and \eqref{eq.r} are not variational, i.e., they cannot be obtained by minimizing some  Rayleigh quotient on a suitable space. For this reason, it is very interesting to study also the natural variational minimization problem related to Dirichlet, Neumann, regional Neumann and Robin boundary conditions. In order to not extend considerably the length of this introduction,  we anticipate that the corresponding minimizers exist, are well defined (see Proposition \eqref{exist.minim}) and   are denoted as $\Lambda_D$, $\Lambda_N$, $\Lambda_{\tilde N}$ and $\Lambda_R$, respectively, but we will not define them here (see equations \eqref{m.d}, \eqref{m.n}, \eqref{m.n2} and \eqref{m.r} for the precise definition).

In spite of the fact that eigenvalues and minimizers are different quantities in general, in light of Proposition \ref{relacion3} they are comparable, with equality in the case of powers (i.e., when $G(t)=t^p/p, p>1)$. Regarding the relation among the different minimizers, in Proposition \ref{relacion} we prove that they are ordered as
$$
\Lambda_{\tilde N} \leq \Lambda_N \leq \Lambda_R \leq \Lambda_D.
$$
In view of the aforementioned Proposition \ref{relacion3}, eigenvalues are consequently ordered as
$$
\lambda_{\tilde N}\leq c^2\lambda_N  \leq c^4 \lambda_R\leq c^6 \lambda_D,
$$
where $c=p^+/p^-$.

In Theorem \ref{teo1} we prove that a function reaching the minimization problem for $\Lambda\in \{\Lambda_{\tilde N}, \Lambda_N, \Lambda_R, \Lambda_D\}$ is an eigenfunction for $\lambda \in   \{\lambda_{\tilde N}, \lambda_N, \lambda_R, \lambda_D\}$, respectively. A considerable difference with the case of powers is that, due to the non-homogeneous nature of the problems,  both eigenvalues and minimizers strongly depend on the energy level: for each $\mu>0$, if the eigenfunction/minimizing function is normalized such that $\int_\Omega G(|u|)=\mu$, then $\Lambda$ and $\lambda$ depend on $\mu$. Nevertheless, in Proposition \ref{relacion2} we prove that $\Lambda$ and $\lambda$ are uniformly bounded by below independently of $\mu$.

Before concluding this introduction, we mention some interesting issues we not deal and let as open questions: to establish positivity of eigenfunctions, to obtain its boundedness,  and to study its interior/up to the boundary regularity.

This paper is organized as follows. In Section \ref{sec.prel} we introduce some  preliminary results and definitions, as well as a proof of an integration by parts formula related to the operator $(-\Delta)^s_g$. Section \ref{sec.3} deals with the proof of our existence results. Some  examples of nonlinearities which illustrate Theorems \ref{Three solution} and \ref{teo2} are given in Section \ref{sec.ejemplos}. Finally, Section \ref{sec.5} is devoted to study the eigenvalue problems \eqref{eq.d},\eqref{eq.n},\eqref{eq.n2} and \eqref{eq.r}.

\section{Preliminaries} \label{sec.prel}
In this section we introduce the classes of Young function and fractional Orlicz-Sobolev functions, the
suitable class where the fractional $g$-Laplacian is well defined.
\subsection{Young functions}
An application $G\colon\RR_+\to \RR_+$ is said to be a  \emph{Young function} if it admits the integral formulation $G(t)=\int_0^t g(\tau)\,d\tau$, where the right continuous function $g$ defined on $[0,\infty)$ has the following properties:
\begin{align*}
&g(0)=0, \quad g(t)>0 \text{ for } t>0 \label{g0} \tag{$g_1$}, \\
&g \text{ is nondecreasing on } (0,\infty) \label{g2} \tag{$g_2$}, \\
&\lim_{t\to\infty}g(t)=\infty  \label{g3} \tag{$g_3$} .
\end{align*}
From these properties it is easy to see that a Young function $G$ is continuous, nonnegative, strictly increasing and convex on $[0,\infty)$.

We will assume the following growth behavior on Young functions
\begin{equation} \label{cond} \tag{L}
1<p^-\leq \frac{tg(t)}{G(t)} \leq p^+<\infty \quad \forall t>0
\end{equation}
where $p^\pm$ are fixed numbers. Roughly speaking, condition \eqref{cond} indicates that $G$ remains between two power functions.

The following  properties on Young functions are well-known. See for instance \cite{FJK} for a proof.

\begin{lemma} \label{lema.prop}
Let $G$ be a Young function satisfying \eqref{cond} and $a,b\geq 0$. Then
\begin{align*}
  &\min\{ a^{p^-}, a^{p^+}\} G(b) \leq G(ab)\leq   \max\{a^{p^-},a^{p^+}\} G(b),\tag{$L_1$}\label{L1}\\
  &G(a+b)\leq \mathbf{C} (G(a)+G(b)) \quad \text{with } \mathbf{C}:=  2^{p^+},\tag{$L_2$}\label{L2}\\
	&G \text{ is Lipschitz continuous}. \tag{$L_3$}\label{L_3}
 \end{align*}
\end{lemma}
Condition \eqref{L2} is known as the \emph{$\Delta_2$ condition} or \emph{doubling condition} and, as it is showed in \cite[Theorem 3.4.4]{FJK}, it is equivalent to the right hand side inequality in \eqref{cond}.

The \emph{complementary Young function} $\tilde G$ of a Young function $G$ is defined as
$$
\tilde G(t):=\sup\{tw -G(w): w>0\}.
$$
From this definition the following Young-type inequality holds
\begin{equation} \label{Young}
ab\leq G(a)+\tilde G(b)\qquad \text{for all }a,b\geq 0,
\end{equation}
and the following H\"older's type inequality
$$
\int_\Omega |uv|\,dx \leq \|u\|_G \|v\|_{\tilde G}
$$
for all $u\in L^G(\Omega)$ and $v\in L^{\tilde G}(\Omega)$. Moreover, it is not hard to see that $\tilde G$ can be written in terms of the inverse of $\phi$ as
\begin{equation} \label{xxxx}
\tilde G(t)=\int_0^t g^{-1}(\tau)\,d\tau,
\end{equation}
see \cite[Theorem 2.6.8]{RR}.

Since $\phi^{-1}$ is increasing, from \eqref{xxxx} and \eqref{cond.intro} it is immediate the following relation.
\begin{lemma} \label{lemita.1}
Let $G$ be an Young function satisfying \eqref{cond.intro} such that $g=G'$ and denote by  $\tilde G$ its complementary function. Then
$$
\tilde G(g(t)) \leq (p^+ +1) G(t)
$$
holds for any $t\geq 0$.
\end{lemma}

The following convexity property will be useful.
\begin{lemma} \cite[Lemma 2.1]{Lamperti}\label{lemita}
Let $G$ be a Young function satisfying \eqref{cond.intro} and \eqref{G2}. Then for every $a,b\in \RR$,
$$
\frac{G(|a|) + G(|b|)}{2} \geq
    G\left(\left|\frac{a+b}{2} \right| \right) + G\left(\left|\frac{a-b}{2} \right| \right).
$$
\end{lemma}

\subsection{Fractional Orlicz-Sobolev spaces}\label{sec.sobolev}
Given a Young function $G$, a parameter $s\in(0,1)$ and an open and bounded set $\Omega\subseteq \RR^n$ we consider the spaces
\begin{align*}
&L^G(\Omega) :=\left\{ u\colon \Omega \to \RR \text{ measurable }\colon  \Phi_{G}(u) < \infty \right\},\\
&W^{s,G}(\Omega):=\left\{ u\in L^G(\Omega) \colon \Phi_{s,G,\RR^n}(u)<\infty \right\},\\
&W^{s,G}_{reg}(\Omega):=\{u\in L^G(\Omega)  \colon \Phi_{s,G,\Omega}(u)<\infty\}
\end{align*}
where the modulars $\Phi_G$ and $\Phi_{s,G}$ are defined as
\begin{align*}
&\Phi_{G,\Omega}(u):=\int_{\Omega} G(|u(x)|)\,dx\\
&\Phi_{s,G,\RR^n}(u):=
  \iint_{\RR^n\times\RR^n} G( |D_su(x,y)|)  \,d\mu,\\
&\Phi_{s,G,\Omega}(u):=
  \iint_{\Omega\times\Omega} G( |D_su(x,y)|)  \,d\mu,
\end{align*}
and  the \emph{$s-$H\"older quotient} is defined as
$$
D_s u(x,y):=\frac{u(x)-u(y)}{|x-y|^s},
$$
with $d\mu(x,y):=\frac{ dx\,dy}{|x-y|^n}$.
These spaces are endowed with the so-called \emph{Luxemburg norms}
\begin{align*}
&\|u\|_{L^G(\Omega)} := \inf\left\{\lambda>0\colon \Phi_{G,\Omega}\left(\frac{u}{\lambda}\right)\le 1\right\},\\
&\|u\|_{W^{s,G}(\Omega)} := \|u\|_{L^G(\Omega)} + [u]_{W^{s,G}(\RR^n)},\\
&\|u\|_{W^{s,G}_{reg}(\Omega)} := \|u\|_{L^G(\Omega)} + [u]_{W^{s,G}_{reg}(\Omega)},
\end{align*}
where the  {\em $(s,G)$-Gagliardo semi-norms} are defined as
\begin{align*}
&[u]_{W^{s,G}(\RR^n)} :=\inf\left\{\lambda>0\colon \Phi_{s,G,\RR^n}\left(\frac{u}{\lambda}\right)\le 1\right\},\\
&[u]_{W^{s,G}_{reg}(\Omega)} :=\inf\left\{\lambda>0\colon \Phi_{s,G,\Omega}\left(\frac{u}{\lambda}\right)\le 1\right\}.
\end{align*}
The space $W^{s,G}(\Omega)$ is a reflexive Banach space. Moreover $C_c^\infty$ is dense in $W^{s,G}(\RR^n)$. See \cite[Proposition 2.11]{FBS} and \cite[Proposition 2.9]{DNFBS} for details.

We also consider the following space
$$
W^{s,G}_0(\Omega):=\left\{ u\in W^{s,G}(\RR^n) :\ u=0\ a.e.\ \text{in}\ \RR^n\setminus\Omega\right\}.
$$
Observe that $W^{s,G}_0(\Omega)\subset W^{s,G}(\RR^n)\subset L^{G}(\RR^n)$.

In order to state some embedding results for fractional Orlicz-Sobolev spaces we introduce the following notation.

Given two Young functions $A$ and $B$, we say that \emph{$B$ is essentially stronger than $A$} or equivalently that \emph{$A$ decreases essentially more rapidly than $B$}, and denoted by $A\prec \prec B$, if for each $a>0$ there exists $x_a\geq 0$ such that $A(x)\leq B(ax)$ for $x\geq x_a$.

When the Young function $G$ fulfills condition \eqref{G3}, the critical function for the fractional Orlicz-Sobolev embedding is given by
$$
G_{*}^{-1}(t)=\int_{0}^{t}\frac{G^{-1}(\tau)}{\tau^{\frac{n+s}{n}}}d\tau.
$$

The following result can be found in \cite{sabri2}. See also \cite{Cianchi} for further generalizations.
\begin{theorem}\label{ceb}
	Let $G$ be a Young function satisfying \eqref{G3} and $s\in(0,1)$. Let $\Omega\subset \RR^n$ be a $C^{0,1}$ bounded open subset. Then
\begin{itemize}
  \item[(i)] \label{7}
  the embedding $W^{s,G}_{reg}(\Omega)\hookrightarrow L^{G_{*}}(\Omega)$ is continuous;

  \item[(ii)] \label{Bem}
  for any Young function $B$ such that $B \prec \prec G_{*}$, the embedding  $W^{s,G}_{reg}(\Omega)\hookrightarrow L^{B}(\Omega)$ is compact.
\end{itemize}
\end{theorem}

From \eqref{lema.prop} it follows the following relation between modulars and norms. See \cite[Lemma 3.1]{Sabri1} or \cite[Lemma 2.1]{Fukagai}.
\begin{lemma}\label{ineq1}
   Let $G$ be a Young function satisfying \eqref{cond.intro} and let $\xi^-(t)=\min\{t^{p^-},t^{p^+}\}$, $\xi^+(t)=\max\{t^{p^-},t^{p^+}\}$, for all $t\geq0$.
    Then, given $\Omega\subset \RR^n$, 
    \begin{itemize}
      \item[(i)] $\xi^-(\|u\|_G)\leq\Phi_G(u)\leq\xi^+(\|u\|_G)\ \text{for}\ u\in L^{G}(\Omega)$,
      \item[(ii)] $\xi^-([u]_{s,G})\leq \Phi_{s,G}(u) \leq\xi^+([u]_{s,G})\ \text{for}\ u\in W^{s,G}(\Omega)$.
    \end{itemize}
 \end{lemma}

\subsection{The fractional $g$-Laplacian operator}

Let $G$ be a Young function such that $G'=g$ and $s\in(0,1)$. As anticipated, the \emph{fractional $g-$Laplacian operator} is defined as
$$
(-\Delta_g)^s u :=2 \,\text{p.v.} \int_{\RR^n} g( |D_s u|) \frac{D_s u}{|D_s u|} \frac{dy}{|x-y|^{n+s}},
$$
where p.v. stands for {\em in principal value}. This operator  is well defined between $W^{s,G}(\RR^n)$ and its dual space $W^{-s,G^*}(\RR^n)$. In fact, in \cite[Theorem 6.12]{FBS} the following representation formula  is provided
$$
\langle (-\Delta_g)^s u,v \rangle =   \iint_{\RR^n\times\RR^n} g(|D_s u|) \frac{D_s u}{|D_s u|}  D_s v \,d\mu,
$$
for any $v\in W^{s,G}(\RR^n)$.

On the other hand, the \emph{censored or regional fractional $g-$Laplacian} is well defined between $W^{s,G}_{reg}(\Omega)$ and its dual space and it is defined as
$$
(-\Delta_g)_\Omega^s u :=2 \,\text{p.v.} \int_\Omega g( |D_s u|) \frac{D_s u}{|D_s u|} \frac{dy}{|x-y|^{n+s}},
$$
which acts as
$$
\langle (-\Delta_g)_\Omega^s u,v \rangle =   \iint_{\Omega\times\Omega} g(|D_s u|) \frac{D_s u}{|D_s u|}  D_s v \,d\mu,
$$
for any $v\in W^{s,G}_{reg}(\Omega)$.

\subsection{Integration by parts formula} \label{sec.partes}
Here we prove an integration by parts formula in our settings which exploits the divergence form of the operator. We introduce the following notation
$$
\langle (-\Delta_g)^s u,v \rangle_* =\frac12   \int_{\mathbb{R}^{2n} \setminus (\Omega^c)^2} g(|D_s u|) \frac{D_s u}{|D_s u|}  D_s v \,d\mu,
$$
the modular
$$
\Phi_{s,G,*}(u)=\iint_{\RR^{2n}\setminus (\Omega^c)^2} G(|D_s u(x,y)|)\,d\mu
$$
and the corresponding Luxemburg semi-norm
$$
[u]_{W^{s,G}_*(\RR^n)} = \inf\left\{ \lambda>0\colon  \Phi_{s,G,*}\left(\frac{u}{\lambda}\right) \leq 1 \right\}.
$$
Of course, it is naturally defined the space
$$
W^{s,G}_*(\Omega):=\{u\in L^G(\Omega)  \colon \Phi_{s,G,*}(u)<\infty\}.
$$

\begin{proposition}\label{integration.by.parts}
Given  $u\in \X$, the following holds.
\begin{itemize}
\item[(i)]
The following version of the divergence theorem  is true
$$
\int_\Omega (-\Delta_g)^s u = - \int_{\mathbb{R}^n \setminus \Omega} \mathcal{N}_g u.
$$

\item[(ii)] More generally, we have the following integration by parts formula
$$
\langle (-\Delta_g)^s u,v \rangle_{*} = \int_{\Omega} v  (-\Delta_g)^s u \ dx + \int_{\mathbb{R}^n \setminus \Omega} v  \mathcal{N}_g u \ dx  \quad  \forall v\in \X.$$
\end{itemize}
\end{proposition}
\begin{proof}
In light of \cite{DNFBS}[Proposition 2.9], it suffices with proving the result for $u\in C_c^2(\RR^n)$.

Let us prove $(i)$. Observe that, since the role of $x$ and $y$ are symmetric, we get
$$
\int_\Omega \int_\Omega g\left( |D_s u|\right)\frac{u(x)}{|D_s u|} \frac{dxdy}{|x-y|^{n+s}} =
\int_\Omega \int_\Omega g\left( |D_s u|\right)\frac{u(y)}{|D_s u|} \frac{dxdy}{|x-y|^{n+s}}
$$
from where it is immediate that
$$
\int_\Omega \int_\Omega g\left( |D_s u|\right)\frac{D_s u}{|D_s u|} \frac{dxdy}{|x-y|^{n+s}} =0.
$$
Hence, we have that
\begin{align*}
\int_\Omega (-\Delta_g)^s u(x) \,dx  &=   \int_\Omega \int_{\mathbb{R}^n} g\left( |D_s u|\right)\frac{D_s u}{|D_s u|} \frac{dy dx}{|x-y|^{n+s}}\\
&=   \int_\Omega \int_{\mathbb{R}^n\setminus  \Omega} g\left( |D_s u|\right)\frac{D_s u}{|D_s u|} \frac{dy dx}{|x-y|^{n+s}}\\
&=    \int_{\mathbb{R}^n\setminus  \Omega} \left(\int_\Omega g\left( |D_s u|\right)\frac{D_s u}{|D_s u|} \frac{dx}{|x-y|^{n+s}} \right)dy\\
&=-\int_{\mathbb{R}^n\setminus \Omega} \mathcal{N}_gu(y)\,dy.
\end{align*}
as desired. Now, let us prove (ii). Since $\mathbb{R}^{2n} \setminus (\Omega^c)^2 = (\Omega\times \mathbb{R}^n) \cup [(\mathbb{R}^n\setminus \Omega) \times \Omega]$, we get
\begin{align*}
\langle (-\Delta_g)^s u,v \rangle_*&= \int_{\Omega} v(x) \left(\int_{\mathbb{R}^n}  g(|D_s u|) \frac{D_s u}{|D_s u|}   \frac{dy}{|x-y|^{n+s}}\right)dx\\
&\quad +  \int_{\mathbb{R}^n\setminus \Omega} v(x) \left( \int_\Omega  g(|D_s u|) \frac{D_s u}{|D_s u|}  \frac{dy}{|x-y|^{n+s}} \right)dx.
\end{align*}
In light of \eqref{g.laplacian} and \eqref{normal} we obtain the desired relation.
\end{proof}

\begin{remark}
If we consider the function  $w_{s,\Omega}(x)=\int_\Omega \int_{\RR^n\setminus \Omega} g(|x-y|^{-s})|x-y|^{n+s}\,dy$ and the normalization of $\mathcal{N}_g$ given by $ \tilde{\mathcal{N}}_g(x):=\frac{\mathcal{N}_g(x)}{w_{s,\Omega}(x)}$, if $\tilde{\mathcal{N}}_g(x)=1$ for any $x\in \RR^n \setminus \bar \Omega$, we can define a generalization of the fractional perimeter defined in \cite{CRS}  as follows
$$
\int_{\RR^n\setminus \Omega} \mathcal{N}_g \,dx = \int_{\RR^n\setminus \Omega} w_{s,\Omega}\,dx = \int_\Omega \int_{\RR^n\setminus \Omega} g\left(\frac{1}{|x-y|^s}\right)\frac{dxdy}{|x-y|^{n+s}}:=\emph{Per}_{s,g}(\Omega).
$$
\end{remark}

\section{Variational setting and proofs of Theorems \ref{Three solution} and \ref{teo2}} \label{sec.3}

We start defining the notion of weak solution for problem \eqref{eq}. With that end it will be useful introducing the following  functional settings. Let us denote
$$
\mathcal{X }:= \{ u\colon \mathbb{R}^n\to \mathbb{R} \text{ measurable s.t.} \colon \|u\|_\mathcal{X} <\infty  	\}
$$
where
$$
\|u\|_\mathcal{X }:=[u]_{W^{s,G}_*(\RR^n)}+ \|u\|_{L^G(\Omega)} +\| u\|_{L^{G,\beta}(\Omega^c)},
$$
and
$$
\| u\|_{L^{G,\beta}(\Omega^c)}= \inf\left\{ \lambda>0\colon \int_{\mathbb{R}^n\setminus\Omega} \beta\ G\left(\frac{u}{\lambda}\right)\,d\mu \leq 1 \right\}.
$$

By following standard arguments it can be seen  that $\mathcal{X}$ is a reflexive Banach space with respect to the norm $\|\cdot\|_{\mathcal{X}}$. See for instance \cite{DNFBS}.

The integration by parts formula given in Proposition \ref{integration.by.parts} leads to the following definition.
\begin{definition}
  We say that $u\in \mathcal{X }$ is a \emph{weak solution} of \eqref{eq} if
$$
\langle (-\Delta_g)^s u,v \rangle_* +\int_{\Omega} g(|u|)\frac{u}{|u|}v\, dx= \lambda\int_\Omega fv\, dx+\mu \int_\Omega hv\, dx - \int_{\mathbb{R}^n\setminus \Omega} \beta g(|u|)\frac{u}{|u|}v\, dx
$$
for all $v\in \mathcal{X}$.
\end{definition}

As anticipated in the introduction,  we will approach problem \eqref{eq} through the machinery of variational methods,  and in particular, it will be done by using the abstract multiplicity result given in Theorem \ref{Ricceri}. With that aim, we  consider the functional
$\Psi\colon \mathcal{X} \to \mathbb{R}$  defined as
$$
  \Psi(u):=\mathcal{J}(u)-\lambda \mathcal{F}(u)-\mu \mathcal{H}(u)
$$
for every $u\in \mathcal{X}$, where  $\mathcal{J}, \mathcal{F}, \mathcal{H}\colon \mathcal{X} \to \mathbb{R}$  are defined as
$$
\mathcal{J}(u):=\int_{\mathbb{R}^{2n} \setminus (\Omega^c)^2} G(|D_s u|)\,d\mu+ \int_{\Omega}G(|u|)dx +\int_{\mathbb{R}^{n}\backslash\Omega}\beta G(|u|)dx,
$$
$$
\mathcal{F}(u)=\int_{\Omega}F(x,u)dx\ \ \text{and}\ \ \mathcal{H}(u)=\int_{\Omega}H(x,u)dx.
$$

The following compact embedding for the space $\X$ holds.
\begin{lemma}\label{embedding}
Given a Young function $A$ such that $A\prec\prec G_*$, then the embedding $\X\hookrightarrow L^A(\Omega)$ is compact.
\end{lemma}

\begin{proof}
  Let $u\in \X$. Observe that $[u]_{W^{s,G}_{reg}(\Omega)}\leq [u]_{W^{s,G}_*(\RR^n)}$ as a consequence of the inequality	
  $$
  \int_{\Omega\times\Omega}G\left(\frac{|D_s u|}{\|u\|_*}\right)\,d\mu\leq \int_{\mathbb{R}^{2n} \setminus (\Omega^c)^2} G\left(\frac{|D_s u|}{\|u\|_*}\right)\,d\mu\leq 1
  $$
  together with the definition of the Luxemburg norm. Then, from Theorem \ref{ceb}, there exists a constant $c>0$ such that
  $$
  \|u\|_{L^A(\Omega)}\leq c[u\|_{W^{s,G}_{reg}(\Omega)} \leq c([u]_{W^{s,G}_*(\RR^n)} + \|u\|_{L^G(\Omega)})\leq c\|u\|_\X
  $$
  concluding the proof due to the compactness of $W^{s,G}_{reg}(\Omega)$ into $L^{A}(\Omega)$.
\end{proof}

The next proposition proves the well-posedness of $\Psi$.
\begin{proposition}
Let $f,h\in \mathcal{A}$, then the functional $\Psi$ is well defined on the space $\mathcal{X}$.
\end{proposition}

\begin{proof}
First, we notice that given $u\in \mathcal{X}$, from Lemma \ref{ineq1} it follows that $\mathcal{J}(u)\leq C \xi^+(\|u\|_\mathcal{X})$ for some constant $C=C(p^\pm)$. Moreover, by \eqref{f1} and the fact that $m$ is increasing we get
$$
\int_{\Omega}F(x,u)\,dx \leq \int_\Omega w(x)\int_0^u  m(|t|)\,dt\,dx \leq  \|w\|_\infty    \int_\Omega |u| m(|u|)\,dx.
$$
In light of Lemma \ref{lemita.1}, $m(|u|)\in L^{\tilde M}(\Omega)$, and then, by applying H\"older's inequality for Young function we get that
$$
\int_\Omega |u| m(|u|)\,dx \leq \|u\|_{L^M(\Omega)} \|m(u)\|_{L^{\tilde M}(\Omega)}.
$$
Observe that \cite[Theorem 3.17.1]{FJK} and Lemma \ref{lemita.1} give that $\|m(u)\|_{L^{\tilde M}(\Omega)}\leq c \|u\|_{L^M(\Omega)}$. Moreover, from Lemma \ref{embedding} it follows that $\|u\|_{L^M(\Omega)} \leq c \|u\|_\X$, and therefore $\mathcal{F}$ is well defined.

The well-posedness of $\mathcal{H}$ follows analogously, concluding that $\Psi$ is well defined on $\X$.
\end{proof}

Next, we prove some useful properties of the functional $\mathcal{J}$.

\begin{lemma}\label{Phit} Assume that \eqref{cond.intro}, \eqref{G2} and \eqref{G3} hold. Then,
   \begin{enumerate}
     \item [(i)] the functional $\mathcal{J}: \mathcal{X}\rightarrow \mathbb{R} $ is $C^{1}$ with derivative given by
  $$
  \langle\mathcal{J}'(u),v\rangle=\langle(-\Delta_{g})^{s}u,v\rangle_*+\int_{\Omega} g(|u(x)|)\frac{u}{|u|}v(x)dx+\int_{\mathbb{R}^{n}\backslash\Omega}\beta\ g(|u(x)|)v(x)dx
  $$
  for all $u,v\in \mathcal{X}$;

  \item [(ii)] $\mathcal{J}$ is coercive, sequentially weakly lower semicontinuous;

  \item [(iii)] $\mathcal{J}\in \mathcal{W}_{\mathcal{X}}$, where the class $\mathcal{W}_{\mathcal{X}}$ is given in Definition \eqref{wx};

  \item [(iv)] $\mathcal{J}$ is bounded on each bounded subset of $\X$ and its derivative admits a continuous inverse on $\mathcal{X}^{*}$.
   \end{enumerate}
\end{lemma}

\begin{proof}
$(i)$ From \cite[Proposition 4.1]{Salort2}, it is easy to see that $\mathcal{J}$ is   class $C^{1}$.

\medskip

$(ii)$ Let $u\in\mathcal{X}$ with $\|u\|_{\mathcal{X}}>1$. In view of Lemma \ref{ineq1}, $\mathcal{J}$ is coercive since
$$
\mathcal{J}(u)   \geq \xi^-(\|u\|_{L^G(\Omega)}) +  \xi^-([u]_{W^{s,G}_*(\RR^n)})  +  \xi^-(\| u\|_{L^{G,\beta}(\Omega^c)}) \geq c  \xi^-( \|u\|_\X ),
$$
where $c>0$ depends only on $p^\pm$. Moreover, the sequential weak lower semicontinuity of $\mathcal{J}$ follows by \cite[Lemma 19]{sabri2}.

\medskip

$(iii)$ Let $\{u_k\}_{k\in\NN}$ be a sequence in $\X$ such that $u_k\rightharpoonup u$ in $\mathcal{X}$ and $\liminf_{k\rightarrow\infty}\mathcal{J}(u_{k})\leq \mathcal{J}(u)$.
Then, by the sequential weak lower semicontinuity of $\mathcal{J}$ proven in (ii) we get that, up to a subsequence,  $\mathcal{J}(u_{k})\rightarrow \mathcal{J}(u)$ as $k\rightarrow+\infty$. Since $\frac{u_k+u}{2}$ converges weakly to $u$, and modulars are lower semicontinuous with respect to the weak convergence, we get
\begin{equation}\label{eee}
  \mathcal{J}(u)\leq \liminf_{k\rightarrow\infty}\mathcal{J}\left(\frac{u_k+u}{2}\right).
\end{equation}
We assume by contradiction that $u_k$ does not converge to $u$ in $\mathcal{X}$. Hence, there exists $ \varepsilon> 0$ such that $\left\| \frac{u_k+u}{2}\right\|_\mathcal{X}>\varepsilon$.
Then, by Lemma \ref{ineq1}
\begin{equation} \label{des1}
\mathcal{J}\left( \frac{u_k+u}{2} \right)>\xi^-(\varepsilon).
\end{equation}
On the other hand, by applying Lemma \ref{lemita} it follows that
$$
\frac{1}{2} \left( G(u)+G(u_k) \right) -G\left(\frac{u_k+u}{2} \right) \geq G\left(\frac{u_k-u}{2} \right),
$$
which together with \eqref{des1} leads to
$$
\frac{1}{2} \left( \mathcal{J}(u)+\mathcal{J}(u_k) \right) -\mathcal{J}\left(\frac{u_k+u}{2} \right) \geq \mathcal{J}\left(\frac{u_k-u}{2} \right) > \xi^-(\varepsilon).
$$
Taking limsup in the above inequality  we obtain that
$$
\mathcal{J}(u)-\xi^-(\varepsilon) \geq \limsup_{k\rightarrow+\infty}\mathcal{J}\left( \frac{u_k+u}{2}\right),
$$
which contradicts \eqref{eee}. Therefore $u_k\to u$ strongly in $\X$, and then $\mathcal{J}\in \mathcal{W}_{\mathcal{X}}$.\\

$(iv)$ When $\|u\|_{\mathcal{X}}\leq \rho$, in light of Lemma \ref{ineq1} we have that $\mathcal{J}(u)\leq \xi^+(\rho)$, i.e., $\mathcal{J}$ is bounded on any bounded subset of $\mathcal{X}$.

We prove now that $\mathcal{J}$ admits a continuous inverse $\mathcal{J}^{-1}\colon \X^* \to \X$ by means of   the monotone operator method introduced by Browder and Minty (see \cite[Theorem 26.A (d)]{Zeidler}). Therefore, it suffices to verify that  $\mathcal{J}'$ is coercive, hemicontinuous and uniformly monotone.

Observe that since $G$ is convex, $\mathcal{J}$ also is convex. Thus $\mathcal{J}(u)\leq \langle\mathcal{J}^{'}(u),u\rangle$ for all $u\in \X$, and, by using Lemma \ref{ineq1}, for any $u\in\X$ such that $\|u\|_{\X}>1$ we have
$$
\frac{\langle\mathcal{J}'(u),u\rangle}{\|u\|_{\mathcal{X}}}\geq \frac{\mathcal{J}(u)}{\|u\|_{\mathcal{X}}}\geq \min\{\|u\|_{\mathcal{X}}^{p^- -1},\|u\|_{\mathcal{X}}^{p^+ -1}\},
$$
from where the coercivity of $\mathcal{J}'$ follows by taking $\|u\|_\X \to \infty$.

Furthermore, since the real function $t\mapsto \langle\mathcal{J}'(u+tv),w\rangle$ is continuous in $[0, 1]$ for any $u,v,w\in \mathcal{X}$, we have that $\mathcal{J}'$ is hemicontinuous.

Let us finally prove that $\mathcal{J}'$ is uniformly monotone. Since $G$ is convex we have that for every $u,v\in\X$ it holds
$$
G(|u|)\leq G\left(\left|\frac{u+v}{2}\right|\right)+g(|u|)\frac{u}{|u|} \frac{u-v}{2} \quad \text{ and } \quad
G(|v|)\leq G\left(\left|\frac{u+v}{2}\right|\right)+g(|v|)\frac{v}{|v|} \frac{v-u}{2}.
$$
Adding the above two relations and integrating over $\Omega$ we find that
\begin{align*}
 \frac{1}{2}\int_{\Omega}& \left(g(|u|)\frac{u}{|u|}-g(|v|) \frac{v}{|v|}\right)(u-v)\, dx\nonumber\\
 &\geq \int_{\Omega}G(|u|)\, dx+\int_{\Omega}G(|v|)\ dx-2\int_{\Omega}G\bigg{(}\bigg{|}\frac{u+v}{2}\bigg{|}\bigg{)}\, dx\quad \forall u,v\in\X.
\end{align*}
On the other hand, we deduce by Lemma \ref{lemita} that
\begin{align*}
  \int_{\Omega}\left( G(|u|)+G(|v|) \right)\,dx
  \geq 2\int_{\Omega}G\left(\frac{u+v}{2} \right)\, dx+ 2\int_{\Omega} G\left(\frac{u-v}{2} \right)\,dx \quad \forall u,v\in\X.
\end{align*}
From the last two relations it follows that
\begin{align*}
  \int_{\Omega}(g(|u|)\frac{u}{|u|}&-g(|v|)\frac{v}{|v|})(u-v)\, dx\geq 4\int_{\Omega} G\left(\frac{|u-v|}{2} \right)\, dx \quad \forall u,v\in\X.
\end{align*}
Similarly, for any $u,v\in \X$ it holds that
\begin{align*}
  \int_{\mathbb{R}^n\backslash\Omega}\mathbf{\beta}\ (g(|u|)\frac{u}{|u|}&-g(|v|)\frac{v}{|v|})(u-v)\ dx\geq 4\int_{\mathbb{R}^n\backslash\Omega}\mathbf{\beta}\  G\left(\frac{|u-v|}{2} \right)\ dx
\end{align*}
and
\begin{align*}
 \langle (-\Delta_g)^s(u-v),u-v \rangle_*\geq 4\int_{\mathbb{R}^{2n} \setminus (\Omega^c)^2} G\left(\frac{|D_s u-D_s v|}{2} \right)\ d\mu.
\end{align*}
Gathering the last three inequalities one gets that
$$
\langle \mathcal{J}'(u)-\mathcal{J}'(v),u-v \rangle\geq 4\mathcal{J}\left(\frac{u-v}{2}\right) \quad \forall u,v\in\X.
$$
Define now the function $\alpha\colon[0,+\infty)\rightarrow[0,+\infty)$ by
$$
\alpha(t)=\frac{1}{p^+ -2}\begin{cases}
                           t^{p^+ -1} & \mbox{for}\ t\leq1 \\
                           t^{p^- -1} & \mbox{for}\ t\geq1.
                          \end{cases}
$$
It is easy to check that $\alpha$ is an increasing function with $\alpha(0)=0$ and $\alpha(t)\to \infty$ as $t\to\infty$.
Taking into account the above information and Lemma \ref{ineq1}, we deduce that
$$
\langle \mathcal{J}'(u)-\mathcal{J}'(v),u-v \rangle\geq \alpha(\|u-v\|_{\mathcal{X}}),
$$
that is, $\mathcal{J}'$ is uniformly monotone, which concludes our proof.
\end{proof}

\begin{lemma}\label{Jt}
  $\mathcal{F}\colon \mathcal{X}\rightarrow \mathbb{R} $ is $C^{1}$ with derivative given by
  \begin{equation}\label{J'}
    \langle \mathcal{F}'(u),v\rangle=\int_{\Omega}f(x,u)v\,dx,
  \end{equation}
  for all $u,v\in \mathcal{X}$.  Moreover, $\mathcal{F}\colon \mathcal{X}\rightarrow \X^*$ is compact.
\end{lemma}

\begin{proof}
Usual arguments show that $\mathcal{F}  \in C^{1}(\mathcal{X},\mathbb{R})$. In order to verify the compactness of $\mathcal{F}$, let $\{u_k\}_{k\in\NN}\subset \mathcal{X}$ be a bounded sequence. Then up to a subsequence $u_k$ weakly converges in $\X$ to $u\in \X$. Moreover, in light of Lemma \ref{embedding}, $u_k\to u$ strongly in $L^M(\Omega)$ and a.e. in $\Omega$.

Fixed $v\in \mathcal{X}$ with $\|v\|_{\mathcal{X}}\leq 1$, thanks to the H\"{o}lder's inequality for Young functions and the embedding of Lemma \ref{embedding} we have
\begin{align*}
|\langle \mathcal{F}'(u_{k}),v \rangle-\langle \mathcal{F}'(u),v \rangle| & =\left| \int_{\Omega}(f(x,u_{k})-f(x,u))v\, dx \right|\\
&\leq \left\| f(\cdot,u_{k}(\cdot))-f(\cdot,u(\cdot))\right\|_{L^{\tilde M}(\Omega)} \left\|v\right\|_{L^M(\Omega)}\\
&\leq c \left\| f(\cdot,u_{k}(\cdot))-f(\cdot,u(\cdot))\right\|_{L^{\tilde M}(\Omega)}\left\| v\right\|_{\mathcal{X}},
\end{align*}
for some $c>0$. Thus, taking supremum for $\|v\|_{\mathcal{X}}\leq 1$, we get
$$\| \mathcal{F}'(u_{k})- \mathcal{F}'(u)\|_{\mathcal{X}^{*}}\leq c\|f(\cdot,u_{k}(\cdot))-f(\cdot,u(\cdot))\|_{L^{\tilde M}(\Omega)}.$$
Being $f\in \mathcal{A}$ we deduce immediately that
$$
f(x,u_{k}(x))- f(x,u(x))\rightarrow0\ \text{as}\ k\rightarrow\infty,
$$
for almost all $x\in\Omega$ and
\begin{align*}
 |f(x,u_{k}(x))- f(x,u(x))| & \leq |f(x,u_{k}(x))|+| f(x,u(x))|\\
                               & \leq \|w\|_{\infty} (m(|u_k(x)|)+m(|u(x)|)).
\end{align*}
Note that the majorant function in the previous relation is uniformly bounded in $L^{\tilde M}(\Omega)$. Hence, by applying the dominate convergence theorem we get that
$$
\int_\Omega \tilde M(|f(x,u_{k}(x))- f(x,u(x))|)\,dx \to 0 \quad \text{ as } k\to\infty.
$$
Since $M$ satisfies \eqref{cond}, $\tilde M-$mean convergence is equivalent to norm convergence (see \cite[Lemma 3.10.4]{FJK}), that is,
$$
\|f(\cdot,u_{k}(\cdot))-f(\cdot,u(\cdot))\|_{L^{\tilde M}(\Omega)}\rightarrow0 \quad  \text{ as } k\rightarrow\infty.
$$
Therefore $\| \mathcal{F}'(u_{k})- \mathcal{F}'(u)\|_{\mathcal{X}^{*}}\rightarrow0$ as  $k\rightarrow\infty$, giving that  $\mathcal{F}'$ is a compact operator.
\end{proof}

\begin{remark}
Combining Lemmas \ref{Phit} and \ref{Jt}, we deduce that $\Psi \in C^{1}(\mathcal{X},\mathbb{R})$ with the derivative given by
\begin{align*}
  \langle \Psi'(u),v\rangle&=\langle(-\Delta_g)^{s}u,v\rangle_*+\int_{\mathbb{R}^{n}\backslash\Omega}\beta\ g(|u(x)|)\frac{u}{|u|}v(x)\,dx+\int_{\Omega}g(|u(x)|)\frac{u}{|u|}v(x)\,dx\\
  &-\lambda\int_{\Omega}f(x,u(x))v(x)\,dx-\mu\int_{\Omega}g(x,u(x))v(x)\,dx,
  \end{align*}
   for every $v\in \mathcal{X}$. Then,  critical points of $\Psi$ are weak solutions of problem \eqref{eq}.
\end{remark}

Having proved these preliminaries, we are in position to prove our first main theorem.

\begin{proof}[{\bf Proof of Theorem \ref{Three solution}}]
Fix $\lambda,\mu$ and $f,h\in\mathcal{A}$, we check the conditions needed to apply Theorem \ref{Ricceri}.

Fixed $\varepsilon>0$, in light of \eqref{pf1} there exist intervals $I_1=[-r_2,-r_1]$ and $I_2=[r_1,r_2]$ such that
\begin{equation} \label{eqq1}
F(x,t)\leq \varepsilon G(|t|) \quad (x,t)\in \Omega \times \RR\setminus (I_1\cup I_2).
\end{equation}
In $I_1\cup I_2$, $F(x,\cdot)$ is bounded in $\Omega$, then there exist $d>0$ and a Young function $B$ such that $b=B'$, $G \prec\prec B \prec\prec G_*$ and $p^+<b^-$ (here $b^-$ denotes a constant such that $b^- <\frac{tb(t)}{B(t)}$)
for which
$$
F(x,t) \leq d  B(|t|) \quad (x,t)\in \Omega \times  (I_1\cup I_2).
$$
Then, from the inequalities above  we obtain that
\begin{align*}
\frac{\mathcal{F}(u)}{\mathcal{J}(u)}\leq  \frac{\int_\Omega F(x,u)\,dx}{ \int_\Omega G(|u|)\,dx} \leq  \varepsilon + d\frac{\int_\Omega B(|u|)\,dx}{ \int_\Omega G(|u|)\,dx}.
\end{align*}
Observe that, assuming that $\|u\|_\X\leq 1$, from Lemma \ref{ineq1} and \cite[Theorem 3.17.1]{FJK} it holds that
$$
\lim_{u\to 0} \frac{\int_\Omega B(|u|)\,dx}{ \int_\Omega G(|u|)\,dx}\leq \lim_{u\to 0} \frac{\|u\|_{L^B(\Omega)}^{b^-}}{\|u\|_{L^G(\Omega)}^{p^+}}\leq c \lim_{u\to 0}  \|u\|_{L^G(\Omega)}^{b^- -p^+}=0.
$$
From the previous computations it follows that
$$
J_1:=\limsup_{u\to 0} \frac{\mathcal{F}(u)}{\mathcal{J}(u)}\leq \varepsilon.
$$
Moreover, assuming that $\|u\|_{\X}\geq 1$, by using again \eqref{eqq1} and Lemma \ref{ineq1} we get
\begin{align*}
\frac{\mathcal{F}(u)}{\mathcal{J}(u)}&\leq \frac{\int_{\{x\in\Omega\colon|u(x)|\leq r_2\}} F(x,u)\,dx}{\mathcal{J}(u)} + \frac{\int_{\{x\in\Omega\colon|u(x)|> r_2\}} F(x,u)\,dx}{ \int_\Omega G(|u|)\,dx}\\
&\leq  \frac{|\Omega|}{\|u\|_\X^{p^+}} \sup\{ F(x,u(x))\colon  (x,u(x))\in \Omega\times [-r_2,r_2] \}  + \varepsilon
\end{align*}
from where we obtain that
$$
J_2:=\limsup_{\|u\|\to \infty} \frac{\mathcal{F}(u)}{\mathcal{J}(u)}\leq \varepsilon.
$$
Therefore, since $\varepsilon$ is arbitrary we obtain that $\max\{0,J_1,J_2\} =0$.

Finally, since we are assuming \eqref{pf2} it follows that the quantity $\sup\{ \mathcal{F}(u)/\mathcal{J}(u)\colon u \in \mathcal{J}^{-1}([0,\infty])\}$ is strictly positive.

Finally, gathering Lemma \ref{Phit}, Lemma \ref{Jt} and the last computations, we are in position of applying Theorem \ref{Ricceri} to obtain our  conclusion.
\end{proof}

Finally, we prove our second existence result.
\begin{proof}[Proof of Theorem \ref{teo2}]
Since $G\prec\prec G_*$, hypothesis $(i)$ implies \eqref{f1}.

Note that hypothesis $(i)$ also implies that
$$
\frac{\mathcal{F}(u)}{\mathcal{J}(u)}\leq \frac{\int_\Omega F(x,u)\,dx}{\int_\Omega G(|u|)\,dx} \leq c_1\frac{\int_\Omega(1+B(|u|)\,dx )}{\int_\Omega G(|u|)\,dx}.
$$
Assuming that $\|u\|_\X\geq 1$, from Lemma \ref{ineq1} and \cite[Theorem 3.17.1]{FJK} it holds that
$$
\lim_{\|u\|_\X \to \infty} \frac{\int_\Omega B(|u|)\,dx}{ \int_\Omega G(|u|)\,dx}\leq \lim_{u\to \infty} \frac{\|u\|_{L^B(\Omega)}^{b^+}}{\|u\|_{L^G(\Omega)}^{p^-}}\leq c \lim_{u\to \infty}  \|u\|_{L^G(\Omega)}^{b^+ -p^-}=0.
$$
From where
$$
J_1:=\limsup_{\|u\|_\X \to \infty} \frac{\mathcal{F}(u)}{\mathcal{J}(u)} = 0.
$$
Similarly, assuming that $\|u\|_\X \leq 1$, hypothesis $(ii)$ implies that
$$
J_2:=\limsup_{\|u\|_\X \to 0} \frac{\mathcal{F}(u)}{\mathcal{J}(u)}
\leq
c_2 \lim_{\|u\|_\X \to 0} \frac{\int_\Omega D(|u|)\,dx}{ \int_\Omega G(|u|)\,dx} \leq c   \lim_{\|u\|_\X \to 0}  \|u\|_{L^G(\Omega)}^{d^- -p^+}=0.
$$
From these relations it follows that $\max\{ 0 , J_1, J_2\} =0$.

Now, without loss of generality we   assume that $\tau_2>0$ and choose a function $u\in \X$ such that $u(x)\geq 0$ in $\Omega$ and such that there exists $x_0\in\Omega$ with $u(x_0)>\tau_2$. It follows that $\mathcal{U}:=\{ x\in\Omega\colon u(x)>\tau_2\}$ is a nonempty open subset of $\Omega$.

Let $k\colon \RR\to\RR$ defined by $k(t)=\min\{t,\tau_2\}$. Then $k(0)=0$ and $k$ is Lipschitz with Lipschitz constant $1$.
Therefore, the function $u_1 = k\circ u \in \X$ satisfies that $u_1(x)=t$ for every $x\in U$ and $0\leq u_1(x)\leq \tau_2$ for every $x\in\Omega$. Then, by hypothesis $(iii)$ we obtain that
$$
F(x,u_1(x))>0 \quad \text{ for any }x\in \mathcal{U}, \qquad F(x,u_1(x))\geq 0 \quad \text{ for every } x \in \Omega.
$$
From this we conclude that $\mathcal{F}(u_1)>0$ and thus
$$
\delta^{-1} = \sup\left\{ \frac{\mathcal{F}(u)}{\mathcal{J}(u)}\colon u\in \mathcal{J}^{-1}((0,\infty)) \right\}>0.
$$
Therefore, from Lemma \ref{Phit}, Lemma \ref{Jt} and the last computations, the result follows by applying Theorem \ref{Ricceri}.
\end{proof}

\section{Some examples of nonlinearities} \label{sec.ejemplos}

Let $G$ be a Young function satisfying \eqref{cond.intro}, \eqref{G2} and \eqref{G3}. Let us prove that the following examples of nonlinearities belong to the class $\mathcal{A}$ and satisfy the hypothesis of Theorem \ref{Three solution}.
\begin{itemize}
\item[(i)]
Consider the function
$f(t)= p|\sin t|^{p-2}\sin t \cos t$ with $p^+<p<p^+_*$ and observe that
$$
|f(t)|\leq p (1+|t|^{p_*^+ -1}),
$$
and since $F(t)=  |\sin t|^p$ we obtain
$$
\lim_{|t|\to 0} \frac{\sup_x F(t)}{G(t)} \leq \lim_{|t|\to 0} \frac{|\sin t|^p}{|t|^{p^+}}=0, \qquad
\lim_{|t|\to \infty} \frac{\sup_x F(t)}{G(t)} \leq \lim_{|t|\to \infty} \frac{|\sin t|^p}{|t|^{p^-}}=0.
$$
Finally, given a compact set $C\subset \Omega$ of positive measure, we consider a function $v\in \X$ such that $v(x)=\frac{\pi}{2}$ in $C$ and $0\leq v(x)\leq \frac{\pi}{2}$ in $\Omega\setminus C$. Then
$$
\sup_{u\in \X} \int_\Omega F(u)\,dx \geq  \int_\Omega |\sin v(x)|^p \,dx = |C| + \int_{\Omega\setminus C} |\sin v(x)|^p\,dx >0.
$$

\medskip

\item[(ii)] More generally, let $M$ be a Young function such that $p^-<p^+<m^-<m^+$, where
$$
1<m^- < \frac{t m(t)}{M(t)} < m^+ <\infty \quad \text{ for all } t\geq 0.
$$
Consider the function $f(t)= m(|\sin t|) \cos t$ for $t\geq 0$, and observe that this function fulfills that $
|f(t)| \leq \max\{m(1),1\} + m(|t|)$. Moreover,  taking $\tau = \sin r$, we get
$$
\int_0^t m(\sin r) \cos r \,dr = \int_0^t m(\tau)\,d\tau = M(|\sin t|)
$$
from where
$$
\lim_{|t|\to 0} \frac{\sup_x F(t)}{G(t)} \leq   \lim_{|t|\to 0} \frac{M(|\sin t|)}{G(|t|)}=0 \leq \lim_{|t|\to 0}   \frac{|\sin t|^{m^-}}{|t|^{p^+}}=0
$$
and
$$
\lim_{|t|\to \infty} \frac{\sup_x F(t)}{G(t)} \leq  \lim_{|t|\to \infty} \frac{M(|\sin t|)}{G(|t|)}=0 \leq \lim_{|t|\to \infty}  \frac{|\sin t|^{m^+}}{|t|^{p^-}}=0.
$$
As before, given a compact set $C\subset \Omega$ of positive measure, we consider a function $u\in \X$ such that $v(x)=\frac{\pi}{2}$ in $C$ and $0\leq v(x)\leq \frac{\pi}{2}$ in $\Omega\setminus C$. Then
$$
\sup_{u\in \X} \int_\Omega F(u)\,dx \geq  \int_\Omega G(|\sin v(x)|) \,dx = |C| + \int_{\Omega\setminus C} G(|\sin v(x)|)\,dx >0.
$$

\medskip

\item[(iii)] We consider the following concave-convex combination
$$
f(t)= t^{p-1}-t^{q-1} \quad \text{with } p^+<p<q<p^+_*:=\frac{np^+}{n-sp^+}.
$$
Note that for some positive constant $c=c(p^\pm)$ it holds that
$$
|f(t)|\leq c(1+|t|^{p-1}) \leq  c(1+|t|^{p^+_*-1}).
$$
Moreover,
$$
\lim_{|t|\to 0} \frac{\sup_x F(t)}{G(t)} \leq \lim_{|t|\to 0} \frac{\frac{|t|^p}{p} - \frac{|t|^q}{q}}{|t|^{p^+}}=0, \quad
\lim_{|t|\to \infty} \frac{\sup_x F(x,t)}{G(t)} \leq \lim_{|t|\to \infty} \frac{\frac{|t|^p}{p} - \frac{|t|^q}{q}}{|t|^{p^-}}=-\infty.
$$
Finally, let a compact set $C\subset \Omega$ large enough and $v\in \X$ such that $v(x)=\tau$ in $C$ and $0\leq v(x)\leq \tau$ in $\Omega\setminus C$, where $\tau$ is chosen such that $\frac{\tau^q}{q}-\frac{\tau^p}{p}>0$. Then
\begin{align*}
\sup_{u\in \X} \int_\Omega F(u)\,dx &\geq \int_\Omega F(v)\,dx =\frac1q \int_\Omega v^q\,dx - 	\frac1p \int_\Omega v^p\,dx  \\&
\geq \frac1q \int_C v^q\,dx - \frac1p \int_C v^p\,dx  - \frac1p \int_{\Omega\setminus C} v^p\,dx \\
&\geq |C|\left( \frac{\tau^q}{q} -\frac{\tau^p}{p} \right)- \frac{\tau^p}{p}|\Omega\setminus C|>0.
\end{align*}

\end{itemize}
\medskip

The following example satisfies the hypothesis of Theorem \ref{teo2}.
\begin{itemize}
\item[(iv)] Let $0<\alpha<p^- \leq p^+ < \beta$. Consider the function
\begin{align*}
f_1(t)=
\begin{cases}
|t|^{\alpha-2}t &\text{ if } |t|\leq 1\\
|t|^{\beta-2}t &\text{ if } |t|> 1.
\end{cases}
\end{align*}
Then, it easily follows that
\begin{align*}
F_1(t)=
\begin{cases}
\frac{|t|^\alpha}{\alpha} &\text{ if } |t|\leq 1\\
\frac{1}{\alpha}-\frac{1}{\beta} + \frac{1}{\beta} |t|^\beta &\text{ if } |t|>1,
\end{cases}
\end{align*}
and conditions $(i)$--$(iii)$ from Theorem \ref{teo2} are fulfilled.
\end{itemize}

\section{Eigenvalues and minimizers} \label{sec.5}
We start this section by defining the notion of eigenvalues.
\begin{definition}
We say that $\lambda$ is an \emph{eigenvalue} of \eqref{eq.d} with \emph{eigenfunction} $u\in W^{s,G}_0(\Omega)$ if
\begin{equation} \label{eig.d}
\langle (-\Delta_g)^s u,v\rangle  =(\lambda -1)\int_\Omega g(|u|) \frac{u}{|u|} v \,dx \qquad \forall v \in W^{s,G}_0(\Omega).
\end{equation}

We say that $\lambda$ is an \emph{eigenvalue} of \eqref{eq.n} with \emph{eigenfunction} $u\in W^{s,G}_*(\Omega)$ if
\begin{equation} \label{eig.n}
\langle (-\Delta_g)^s u,v\rangle_*   =(\lambda -1)\int_\Omega g(|u|)  \frac{u}{|u|} v \,dx \qquad \forall v \in W^{s,G}_*(\Omega).
\end{equation}

We say that $\lambda$ is an \emph{eigenvalue} of \eqref{eq.n2} with \emph{eigenfunction} $u\in W^{s,G}_{reg}(\Omega)$ if
\begin{equation} \label{eig.n2}
\langle (-\Delta_g)_\Omega^s u,v\rangle   =(\lambda -1) \int_\Omega g(|u|)  \frac{u}{|u|} v \,dx \qquad \forall v \in W^{s,G}_{reg}(\Omega).
\end{equation}

We say that $\lambda$ is an \emph{eigenvalue} of \eqref{eq.r} with \emph{eigenfunction} $u\in \X$ if
\begin{equation} \label{eig.r}
\langle (-\Delta_g)_\Omega^s u,v\rangle_*   =(\lambda -1) \int_\Omega g(|u|)  \frac{u}{|u|} v \,dx  - \int_{\mathbb{R}^n\setminus \Omega} \beta g(|u|)\frac{u}{|u|}v\, dx \qquad \forall v \in\X.
\end{equation}
\end{definition}

\bigskip
In order to prove our eigenvalues and minimizers results we will consider the following functionals defined in Sections \ref{sec.sobolev} and \ref{sec.partes}
$$
\Phi_{s,G,\RR^n}(u)\colon W^{s,G}_0(\Omega)\to \RR, \quad \Phi_{s,G,\Omega}(u)\colon W^{s,G}_{reg}(\Omega)\to \RR, \quad  \Phi_{s,G,*}(u)\colon W^{s,G}_*(\Omega)\to \RR
$$
and
$$
\Phi_G(u)\colon \mathcal{W}(\Omega)\to \RR
$$
where its definition domain $\mathcal{W}(\Omega)$ is either $W^{s,G}_0(\Omega)$, $W^{s,G}_{reg}(\Omega)$ or $W^{s,G}_*(\Omega)$.

Following \cite[Proposition 4.1]{Salort2} it is straightforward to see that these functionals are well-defined and are Fr\'echet derivable. Moreover, the following expressions can be deduced.

\begin{proposition}
We have that $(\Phi_{s,G,\RR^n})'$ is defined from $W^{s,G}_0(\Omega)$ onto its dual, $(\Phi_{s,G,\Omega})'$ from $ W^{s,G}_{reg}(\Omega)$ onto its dual, $(\Phi_{s,G,*})'$ from $W^{s,G}_*(\Omega)$ onto its dual, and $(\Phi_{G,\Omega})'$ from $\mathcal{W}(\Omega)$ onto its dual, are $C^1$ and their Fr\'echet derivatives are given by
\begin{align*}
&\langle (\Phi_{s,G,\RR^n})'(u),v \rangle = \langle (-\Delta_g)^su,v\rangle \qquad  \forall v \in W^{s,G}_0(\Omega),\\
&\langle (\Phi_{s,G,\Omega})'(u),v \rangle = \langle (-\Delta_g)^su,v\rangle \qquad  \forall v \in W^{s,G}_{reg}(\Omega),\\
&\langle (\Phi_{s,G,*})'(u),v \rangle = \langle (-\Delta_g)^su,v\rangle_* \qquad \forall v \in W^{s,G}_*(\Omega),\\
&\langle (\Phi_{G,\Omega})'(u),v \rangle = \int_\Omega g(|u|)\frac{u}{|u|}v\,dx \qquad \forall v \in \mathcal{W}(\Omega).
\end{align*}
\end{proposition}
\begin{proof}
See \cite[Proposition 4.1]{Salort2} with the pertinent changes.
\end{proof}

Given $\mu>0$, we consider the minimization problems
\begin{equation} \label{m.d}
\Lambda_D :=\inf_{u\in M^D_\mu} \frac{\Phi_{s,G,\RR^n}(u) + \Phi_{G,\Omega}(u)}{\Phi_{G,\Omega}(u)} \quad \text{ with }\quad  M^D_\mu =\{ u \in W^{s,G}_0 \colon \Phi_{G,\Omega}(u)=\mu\},
\end{equation}

\begin{equation} \label{m.n}
\Lambda_N :=\inf_{u\in M^N_\mu} \frac{\Phi_{s,G,*}(u) + \Phi_{G,\Omega}(u)}{\Phi_{G,\Omega}(u)} \quad \text{ with }\quad  M^N_\mu =\{ u \in W^{s,G}_*(\Omega) \colon \Phi_{G,\Omega}(u)=\mu\},
\end{equation}

\begin{equation} \label{m.n2}
\Lambda_{\tilde N} :=\inf_{u\in M^{\tilde N}_\mu} \frac{\Phi_{s,G,\Omega}(u) + \Phi_{G,\Omega}(u)}{\Phi_{G,\Omega}(u)} \quad \text{ with }\quad  M^{\tilde N}_\mu =\{ u \in W^{s,G}_{reg}(\Omega)\colon \Phi_{G,\Omega}(u)=\mu\},
\end{equation}
and
\begin{equation} \label{m.r}
\Lambda_R :=\inf_{u\in M^R_\mu} \frac{\Phi_{s,G,*}(u) + \Phi_{G,\Omega}(u) + \Phi_{G, \beta,\Omega^c}(u)}{\Phi_{G,\Omega}(u)}
\end{equation}
with
$$
M^R_\mu =\{ u \in \X\colon  \Phi_{G,\Omega}(u)=\mu\} \quad \text{and} \quad  \Phi_{G,\beta, \Omega^c}(u):=\int_{\RR^n\setminus \Omega}\beta\ G(|u(x)|)\,dx.
$$
Note the subindex refers to Dirichlet, Neumann, regional Neumann and Robin, respectively. Moreover, due to the possible lack of homogeneity, in general the quantities defined above depend on the energy level $\mu$.

\begin{remark}
By using the Poincar\'e's inequality \cite[Proposition 3.2]{Salort2} it follows that $[\cdot]_{}$ is an equivalent norm in $W^{s,G}_0(\Omega)$.
\end{remark}

\begin{proposition} \label{exist.minim}
For each $\mu>0$ there exist solutions of the minimization problems \eqref{m.d}, \eqref{m.n}, \eqref{m.n2} and \eqref{m.r}, respectively.
\end{proposition}
\begin{proof}
It follows just by applying the direct method of the calculus of variations. See \cite[Proposition 5.1]{Salort2}.
\end{proof}

Existence of minimizers allow us to prove existence of eigenvalues.
\begin{theorem} \label{teo1}
For every $\mu>0$ there exist positive numbers $\lambda_D$, $\lambda_N$, $\lambda_{\tilde N}$ and $\lambda_{R}$ which are eigenvalues of \eqref{eq.d}, \eqref{eq.n}, \eqref{eq.n2} and \eqref{eq.r}, respectively, with non-negative eigenfunctions $u_D\in W^{s,G}_0(\Omega)$, $u_{N}\in W^{s,G}_*(\Omega)$, $u_{\tilde N}\in W^{s,G}_{reg}(\Omega)$ and $u_R\in \X$, respectively,  normalized such that $\Phi_{G,\Omega}(u_D)=\Phi_{G,\Omega}(u_N)=\Phi_{G,\Omega}(u_{\tilde N})=\Phi_{G,\Omega}(u_R)=\mu$.
\end{theorem}
\begin{proof}
Given a fixed $\mu>0$, in light of Proposition \ref{exist.minim} there exist functions $u_D\in W^{s,G}_0(\Omega)$, $u_{N}\in W^{s,G}_*(\RR^n)$, $u_{\tilde N}\in W^{s,G}(\Omega)$ and $u_R\in \X$,  respectively,  normalized such that their modular $\Phi_{G,\Omega}$ is equal to $\mu$, which attain the minimization problems \eqref{m.d}, \eqref{m.n}, \eqref{m.n2} and \eqref{m.r}, respectively.

Therefore, by the Lagrange multipliers rule, since the involved functionals are $C^1$, there exist numbers $\lambda_D$, $\lambda_N$, $\lambda_{\tilde N}$ and $\lambda_R$ for which the corresponding function $u_D$, $u_N$, $u_{\tilde N}$ and $u_R$ satisfy the weak formulations \eqref{eig.d}, \eqref{eig.n}, \eqref{eig.n2} and \eqref{eig.r}, respectively.
\end{proof}

\begin{proposition} \label{relacion}
The following relations among the minimizers of \eqref{m.d}, \eqref{m.n} and \eqref{m.n2} holds
$$
\Lambda_{\tilde N} \leq \Lambda_N \leq \Lambda_R \leq \Lambda_D.
$$
\end{proposition}

\begin{proof}
Observe that since $\Phi_{s,G,\Omega}(u)\leq \Phi_{s,G,*}(u)$ it follows that $W^{s,G}_*(\Omega)\subset W^{s,G}_{reg}(\Omega)$. Then, given a minimizer $u\in M_\mu^N$ of $\Lambda_N$ we get that
$$
\min_{u\in W^{s,G}_*(\Omega)} \frac{1}{\mu}\left( \mu +\Phi_{s,G,\Omega}(u) \right)\leq \min_{u\in W_{s,G,*}(\Omega)} \frac{1}{\mu} \left( \mu +\Phi^{s,G}_*(u), \right) = \Lambda_N
$$
but  since minimizing over a small set enlarges the minimum, we conclude that
$$
\Lambda_{\tilde N}=\min_{u\in W^{s,G}_{reg}(\Omega)} \frac{1}{\mu}\left( \mu +\Phi_{s,G,\Omega}(u) \right) \leq \min_{u\in W^{s,G}_*(\Omega)} \frac{1}{\mu}\left( \mu +\Phi_{s,G,\Omega} (u) \right) \leq \Lambda_N.
$$

Moreover,  since $\X \subset W^{s,G}_*(\Omega)$ it follows that $\Lambda_N \leq \Lambda_R$.

Finally, note that $u\in W^{s,G}_0(\Omega)$ if $\Phi_{s,G,*}(u)<\infty$ and $u=0$ in $\RR^n\setminus \Omega$. Therefore, $W^{s,G}_0(\Omega)\subset \X$ and $\Phi_{s,G,\RR^n}(u)=\Phi_{s,G,*}(u)$. Then, proceeding as before, $\Lambda_R\leq \Lambda_D$.
\end{proof}

The following proposition claims that minimizers are uniformly bounded away from zero independently of the energy level.
\begin{proposition} \label{relacion2}
Given $\mu>0$, the minimizers $\Lambda_{\tilde N}$, $\Lambda_N$, $\Lambda_R$ and $\Lambda_D$ are positive and bounded by below independently on $\mu$.
\end{proposition}

\begin{proof}
The Dirichlet case is treated in \cite[Theorem 4.2]{Salort2}. We deal here with the general case.

Given $\mu>0$, let $u\in M_\mu^{\tilde N}$ be a minimizer of $\Lambda_{\tilde N}$, that is, $u\in W^{s,G}_{reg}(\Omega)$ is such that $\Phi_{G,\Omega}(u)=\mu$ and
$$
\Lambda_{\tilde N}= \frac{\Phi_{s,G,\Omega}(u) + \Phi_{G,\Omega}(u)}{\Phi_{G,\Omega}(u)}.
$$
Denote by $\bar u= \frac{1}{|\Omega|}\int_\Omega u(x)\,dx$ the average of $u$ on $\Omega$.  By using the $\Delta_2$ condition we have that
$$
\int_\Omega G(|u|)\,dx \leq \mathbf{C}  \int_\Omega G(|u-\bar u|)\,dx  + \mathbf{C}  \int_\Omega G(|\bar u|)\,dx.
$$
By using Jensen's inequality and \eqref{L1} we get
\begin{align*}
\int_\Omega G(|u-\bar u|)\,dx &= \int_\Omega G\left(\left| \frac{1}{|\Omega|}\int_\Omega (u(x)-u(y)) \,dy \right| \right) \,dx\\
&\leq
\frac{1}{|\Omega|} \int_\Omega  \int_\Omega G(|u(x)-u(y)|) \,dy  \,dx\\
&\leq
\frac{1}{|\Omega|} \int_\Omega  \int_\Omega G\left(\frac{|u(x)-u(y)|}{|x-y|^s} \,\text{diam}\,(\Omega)^s \right) \,dy  \,dx\\
&\leq c(|\Omega|) \Phi_{s,G,\Omega}(u).
\end{align*}
Finally, since again the Jensen's inequality gives
$$
\int_\Omega G(|\bar u|)\leq \int_\Omega G\left( \frac{1}{|\Omega|}\int_\Omega |u(y)|\,dy\right)\,dx \leq \frac{1}{|\Omega|} \int_\Omega \int_\Omega G(|u(y)|)\,dy \,dx =\Phi_{G,\Omega}(u)
$$
we obtain that $\Phi_{G,\Omega}(u) \leq c(\mathbf{C},|\Omega|)(\Phi_{s,G,\Omega}(u) + \Phi_{G,\Omega}(u))$, which implies a lower bound for $\Lambda_{\tilde N}$:
$$
\Lambda_{\tilde N} = \frac{ \Phi_{s,G,\Omega}(u) + \Phi_{G,\Omega}(u)}{\Phi_{G,\Omega}(u)} \geq \frac{1}{c(|\Omega|)}.
$$
In view of Proposition \ref{relacion}, the same lower bound is admissible for $\Lambda_N$, $\Lambda_R$ and $\Lambda_D$.
\end{proof}

The following proposition states that, although eigenvalues and minimizers differ in general, both quantities are indeed  comparable.
\begin{proposition} \label{relacion3}
It holds that
$$
\frac{p^-}{p^+} \Lambda \leq \lambda \leq \frac{p^+}{p^-} \Lambda
$$
where $\Lambda\in \{\Lambda_D,\Lambda_R,\Lambda_N,\lambda_{\tilde N}\}$ and $\Lambda\in \{\lambda_D,\lambda_R,\lambda_N,\lambda_{\tilde N}\}$, respectively.

As a direct consequence, denoting $c=p^+/p^-$, we have
$$
\lambda_{\tilde N}\leq c^2\lambda_N  \leq c^4 \lambda_R\leq c^6 \lambda_D.
$$
\end{proposition}
\begin{proof}
These first chain of inequalities just follow by testing in the definition of eigenvalue with the eigenfunction itself and using the fact that condition \eqref{cond.intro}, for all $t\geq 0$,  relates $tg(t)$ with $G(t)$ up to the constants $p^\pm$.

The second chain of inequalities are obtained just gathering the first one together with Proposition \ref{relacion}.
\end{proof}

As a consequence of Proposition \ref{relacion2} and \ref{relacion3} we obtain a lower bound for eigenvalues.
\begin{theorem}
 $\lambda_D$, $\lambda_R$, $\lambda_N$, $\lambda_{\tilde N}$ are bounded by below by a positive constant independent on $\mu$.
\end{theorem}

\appendix
\section{An abstract existence result}

\begin{definition} \label{wx}
We introduce the following definitions.
  \begin{itemize}
    \item[(i)] If $X$ is a real Banach space, we denote by $\mathcal{W}_{X}$ the class of all functionals $\mathcal{J}\colon X\rightarrow \mathbb{R} $ possessing the following property: if $\{u_{k}\}_{k\in\NN}$ is a sequence in $X$ converging weakly to $u\in X$ and $\liminf_{k\rightarrow\infty}\mathcal{J}(u_{k})\leq \mathcal{J}(u)$, then  $\{u_{k}\}_{k\in\NN}$ has a subsequence converging strongly to $u$.
    \item[(ii)] We say that the derivative of $\mathcal{J}$ admits a continuous inverse on $X^{*}$ we mean that there exists a continuous operator $T\colon X^{*}\rightarrow X$ such that $T(\mathcal{J}(x))=x$ for all $x\in X$.
  \end{itemize}
\end{definition}
The above property is somehow a compactness property, stating the existence of a convergent subsequence of a given sequence.

\begin{theorem}[\cite{Ricceri}]\label{Ricceri}
Let $X$ be a separable and reflexive real Banach space; $\mathcal{J}\colon X\to \mathbb{R} $ a coercive, sequentially weakly lower semicontinuous $C^{1}$ functional, belonging to $\mathcal{W}_{X}$, bounded on each bounded subset of $X$ and whose derivative admits a continuous inverse on $X^{*}$, and $\mathcal{F}\colon X\to \mathbb{R} $ a  $C^{1}$ functional with compact derivative. Assume that $\Psi$ has a strict local minimum $x_{0}$ with $\mathcal{J}(x_{0})=\mathcal{F}(x_{0})=0$. Finally, setting
$$
\alpha=\max \left\{ 0,\ \limsup_{\|x\|\rightarrow+\infty}\frac{\mathcal{F}(x)}{\mathcal{J}(x)},\ \limsup_{\|x\|\rightarrow x_{0}}\frac{\mathcal{F}(x)}{\mathcal{J}(x)}\right\},
$$
$$
\beta=\sup_{x\in \mathcal{J}^{-1}(]0,+\infty[)}\frac{\mathcal{F}(x)}{\mathcal{J}(x)},
$$
and assume $\alpha<\beta$.
Then, for each compact interval $[a,b]\in (\frac{1}{\beta},\frac{1}{\alpha})$ (with the conventions $\frac{1}{0}=+\infty$, $\frac{1}{+\infty}=0$) there exists $\nu > 0$ with the following property: for every $\lambda\in [a,b]$ and every $C^{1}$ functional  $\mathcal{J}\colon X\rightarrow \mathbb{R} $ with compact derivative, there exists $\gamma>0$ such that, for each $\mu\in[0,\gamma]$, the equation
$$
\mathcal{J}'(x)=\lambda \mathcal{F}'(x)+\mu \mathcal{H}'(x)
$$
has at least three solutions whose norms are less than $\nu$.
\end{theorem}


\begin{thebibliography}{99}


\bibitem{Nicola} N. Abatangelo, A remark on nonlocal Neumann conditions for the fractional Laplacian, arXiv:1712.00320.

\bibitem{AHK}
K. B. Ali, M. Hsini, K. Kefi and N. T. Chung, On a nonlocal fractional $p (.,.)$-Laplacian problem with competing nonlinearities. {\it Complex Analysis and Operator Theory}, {\bf 13} (3)(2019), 1377-1399.

\bibitem{Cianchi} A. Alberico, A. Cianchi, L. Pick and L. Slavíková, Fractional Orlicz-Sobolev embeddings, arXiv:2001.05565.

\bibitem{Cianchi2} A. Alberico, A. Cianchi, L. Pick and L. Slavíková, On the limit as $ s\to 0^+ $ of fractional Orlicz-Sobolev spaces, arXiv:2002.05449.

\bibitem{Azroul} E. Azroul, A. Benkirane and M.Srati, Existence of solutions for a nonlocal type problem in fractional Orlicz Sobolev spaces, {\it Advances in Operator Theory}, https://doi.org/10.1007/s43036-020-00042-0.


\bibitem{BR} A. Bahrouni and Radulescu, V. D. (2018). On a new fractional Sobolev space and applications to nonlocal variational problems with variable exponent. {\it Discrete and Continuous Dynamical Systems-S}, 11(3), 379.

\bibitem{Sabri1} S. Bahrouni, H. Ounaies and L. S. Tavares, Basic results of fractional Orlicz-Sobolev space and applications to non-local problems, {\it Topol. Methods Nonlinear Anal,} to appear.

\bibitem{sabri2} S. Bahrouni and H. Ounaies, Embedding theorems in the fractional Orlicz-Sobolev space and applications to non-local problems, {\it Discret and continuous Dynamical systems}, to appear.

\bibitem{sabri3} S. Bahrouni, Infinitely many solutions for problems in fractional Orlicz-Sobolev spaces, {\it Rochy mountain journal of mathematics}, to appear.

\bibitem{sabri4} A. Bahrouni, S. Bahrouni and M. Xiang, On a class of nonvariational problems in fractional Orlicz-Sobolev
spaces, {\it Nonlinear Analysis,} {\bf 190} (2020), 111595.

\bibitem{FBS} J. Fern\'andez Bonder and A. M. Salort, Fractional order Orlicz-Sobolev spaces, {\it Journal of Functional Analysis}, {\bf 277} (2) (2019), 333-367.

\bibitem{Salort3}  J. Fern\'andez Bonder, M. P\'erez LLanos and A. M. Salort, A H\"{o}lder infinity Laplacian obtained as limit of Orlicz fractional Laplacians, arXiv:1807.01669.

\bibitem{CRS} L. Caffarelli, J.M. Roquejoffre and O. Savin, Nonlocal minimal surfaces,  {\it Comm. Pure Appl. Math}, {\bf 63} (2009), 1111-1144.

\bibitem{Vilasi} F. Cammaroto and L. Vilasi, Multiple solutions for a nonhomogeneous Dirichlet problem in Orlicz–Sobolev spaces, {\it Applied Mathematics and Computation}, {\bf 218} (2012), 11518–11527

\bibitem{DPRS} L.M. Del Pezzo, J.D. Rossi and A.M. Salort, Fractional eigenvalue problems that approximate Steklov eigenvalues. {\it Proc. R. Soc. Edinb. Sect. A}, {\bf148} (3)(2018), 499-516.

\bibitem{DPS} L. Del Pezzo and A. M. Salort, The first non-zero Neumann $p$-fractional eigenvalue, {\it Nonlinear Analysis: Theory, Methods and Applications}, {\bf118} (2015), 130-143.

\bibitem{DNFBS} P. De N\'apoli, J. Fern\'andez Bonder and A. Salort, A P{\'o}lya--Szeg{\"o} principle for general fractional Orlicz--Sobolev spaces, {\it Complex Variables and Elliptic Equations}, (2020), 1-23.

\bibitem{Dipierro} S. Dipierro, X. Ros-Oton and E. Valdinoci, Nonlocal problems with Neumann boundary conditions, {\it Rev. Mat. Iberoam}, {\bf 33}
(2) (2017), 377-416.

\bibitem{Winkert} S. El Manouni, H.Hajaiej and P. Winkert, Bounded solutions to nonlinear problems in $\RR^n$ involving the fractional Laplacian  depending on parameters, {\it Minimax Theory and its applications}, {\bf 2} (2) (2017), 265-283.

\bibitem{K} A. El Khalil, On the spectrum of Robin boundary p-Laplacian problem, {\it Moroccan Journal of Pure and Applied Analysis}, {\bf 5} (2) (2019), 279-293.

\bibitem{Fukagai} N. Fukagai, M. Ito and K. Narukawa, Positive solutions of quasilinear elliptic equations with critical Orlicz-Sobolev nonlinearity on $\mathbb{R}^{d}$, \textit{ Funkcialaj Ekvacioj}, {\bf49} (2006), 235-267.


\bibitem{KRV}
U. Kaufmann, J. Rossi and R. Vidal,
Fractional Sobolev spaces with variable exponents and fractional. Electron. {\it J. Qual. Theory Differ. Equ.} 2017, No. 76, 1-10.


\bibitem{FJK} A. Kufner, O. John and S. Fucik, Function spaces(Vol. 3), {\it Springer Science Business Media,} (1979).

\bibitem{Kristaly} A. Krist\'{a}ly, M. Mih\u{a}ilescu and V. R\u{a}dulescu, Two non-trivial solutions for a non-homogeneous Neumann problem: an Orlicz–Sobolev space setting, {\it Proceedings of the Royal Society of Edinburgh}, {\bf 139 A} (2009), 367-379.

\bibitem{Lamperti} J. Lamperti, On the isometries of certain function-spaces. {\it Pacific J. Math} 8.3 (1958): 459-466.

\bibitem{LL} E. Lindgren and P. Lindqvist, Fractional eigenvalues, {\it Calculus of Variations and Partial Differential Equations}, {\bf 49}(1-2) (2014), 795-826.

\bibitem{ML} D. Mugnai and E. P. Lippi, Neumann fractional $p$-Laplacian: Eigenvalues and existence results, {\it Nonlinear Analysis}, {\bf 188} (2019), 455-474.

\bibitem{Dimitri} D. Mugnai, A. Pinamonti and E. Vecchi, Towards a Brezis-Oswald-type result for fractional problems with Robin boundary conditions. {\it Calc. Var}, {\bf 59} (43) (2020). https://doi.org/10.1007/s00526-020-1708-8.

\bibitem{Nguyen} T.C. Nguyen, Three solutions for a class of nonlocalproblems in Orlicz-Sobolev spaces,  {\it Applied Mathematics and Computation}, {\bf 218} (23) (2013), 1257-1269.

\bibitem{RR} M. Rao and Z. Ren,  Applications of Orlicz spaces (Vol. 250). CRC Press. (2002).

\bibitem{Ricceri} B. Ricceri, A further three critical points theorem, {\it Nonlinear Analysis}, {\bf 71} (2009), 4151-4157.

\bibitem{Salort2} A. M. Salort, Eigenvalues and minimizers for a non-standard growth non-local operator, {\it Journal of Differential Equations}, {\bf 268} (9) (2020), 5413-5439.

\bibitem{SV} R. Servadei and E. Valdinoci, On the spectrum of two different fractional operators, {\it Proceedings of the Royal Society of Edinburgh Section A: Mathematics}, {\bf 144} (4) (2014), 831-855.

\bibitem{Bruno} B. Volzone, Symmetrization for fractional Neumann problems, {\it Nonlinear Analysis}, {\bf 147} (2016), 1-25.

\bibitem{W} M. Warma, The fractional Neumann and Robin type boundary conditions for the regional fractional $p$-Laplacian, {\it Nonlinear Differ. Equ. Appl}, {\bf23} (1) (2016).

\bibitem{Zeidler} E. Zeidler, Nonlinear functional analysis and its applications, vol. II/B Springer, (1985).

\end{thebibliography}
\end{document}